\documentclass[10pt,bezier]{article}
\usepackage{amsmath,amssymb,amsfonts,xspace}
\usepackage{graphicx}
\usepackage{tikz}

\newtheorem{prethm}{{\bf Theorem}}

\newenvironment{thm}{\begin{prethm}{\hspace{-0.5
               em}{\bf.}}}{\end{prethm}}

\newtheorem{prepro}{{\bf Proposition}}

\newtheorem{precor}{{\bf Corollary}}

\newenvironment{cor}{\begin{precor}{\hspace{-0.5
               em}{\bf.}}}{\end{precor}}

\newtheorem{preconj}{{\bf Conjecture}}

\newenvironment{conj}{\begin{preconj}{\hspace{-0.5
               em}{\bf.}}}{\end{preconj}}

\newtheorem{preremark}{{\bf Remark}}

\newenvironment{remark}{\begin{preremark}\rm{\hspace{-0.5
               em}{\bf.}}}{\end{preremark}}

\newtheorem{prelem}{{\bf Lemma}}

\newenvironment{lem}{\begin{prelem}{\hspace{-0.5
               em}{\bf.}}}{\end{prelem}}

\newtheorem{prequestion}{{\bf Question}}

\newenvironment{question}{\begin{prequestion}{\hspace{-0.5
               em}{\bf.}}}{\end{prequestion}}

\newtheorem{preproof}{{\bf Proof.}}

\newenvironment{proof}[1]{\begin{preproof}{\rm
               #1}\hfill{$\Box$}}{\end{preproof}}

\renewcommand{\thefootnote}

\title{On the largest real root of independence polynomials of graphs, an ordering on graphs, and starlike trees}
\author{Mohammad Reza Oboudi$^{\,\rm a,b}$ \\
{\footnotesize {\em $^{\rm a}$Department of Mathematics, University of Isfahan,}}\\
{\footnotesize {\em Isfahan 81746-73441, Iran}}\\
{\footnotesize {\em $^{\rm b}$School of Mathematics, Institute for Research in Fundamental Sciences (IPM),}}\\
 {\footnotesize {\em P.O. Box 19395-5746, Tehran, Iran}}}
\begin{document}
\footnotetext{{\em E-mail Addresses}: {\tt mr.oboudi@sci.ui.ac.ir,\,mr\_oboudi@yahoo.com}
(M.R. Oboudi).}
\date{}
\maketitle
%\begin{quote}
%{\small \hfill{\rule{13.3cm}{.1mm}\hskip2cm}
%\textbf{Abstract}\vspace{1mm}
%{\renewcommand{\baselinestretch}{1}
%\parskip = 0 mm
%the energy of a matrix.}}
%\noindent{\small\noindent{\small {\it AMS
%Classification}}:  05C50; 15A18; 15A42
%\vspace{0.1mm}{\it Keywords}: Energy of matrix; Energy of graph;}
 %\vspace{-3mm}\hfill{\rule{13.3cm}{.1mm}\hskip2cm}
%\end{quote}
%\section{Introduction}

\begin{abstract}
Let $G$ be a simple graph of order $n$. An independent set in a
graph is a set of pairwise non-adjacent vertices. The independence
polynomial of $G$ is the polynomial $I(G,x)=\sum_{k=0}^{n} s(G,k)
x^{k}$, where $s(G,k)$ is the number of independent sets  of $G$ of
size $k$ and $s(G,0)=1$. Clearly all real roots of $I(G,x)$ are negative. Let $\xi(G)$ be the largest real root of
$I(G,x)$. Let $H$ be a simple graph. By $G \succeq H$ we mean
that $I(H,x)\geq I(G,x)$ for every $x$ in the interval
$[\xi(G),0]$. We note that $G\succeq H$ implies that $\xi(G)\geq \xi(H)$. Also we let $G\succ H$ if and only if $G\succeq H$ and $I(G,x)\neq I(H,x)$. We prove that for every tree $T$ of order
$n$, $S_n\succeq T\succeq P_n$, where $S_n$ and $P_n$ are the
star and the path of order n, respectively. By
$T=T(n_1,\ldots,n_k)$ we mean a tree $T$ which has a
vertex $v$ of degree $k$ such that $T\setminus
v=P_{n_1-1}+\cdots+P_{n_k-1}$, that is $T\setminus v$ is the disjoint union of the paths $P_{n_1-1},\ldots,P_{n_k-1}$. Let $X=(x_1,\ldots,x_k)$ and
$Y=(y_1,\ldots,y_k)$, where $x_1\geq \cdots\geq x_k$ and
$y_1\geq\cdots\geq y_k$ are real. By $X\succ Y$, we mean
$x_1=y_1,\ldots,x_{t-1}=y_{t-1}$ and $x_t>y_t$ for some $t\in\{1,\ldots,k\}$. We
let $X\succ_{d}Y$, if $X\neq Y$ and for every $j$, $1\leq j\leq k$,
$\sum_{i=1}^{j}x_i\geq \sum_{i=1}^{j}y_i$. Among all
trees with fixed number of vertices, we show that if
$(m_1,\ldots,m_k)\succ_{d}(n_1,\ldots,n_k)$, then
$T(n_1,\ldots,n_k)\succ T(m_1,\ldots,m_k)$. We conjecture that $T(n_1,\ldots,n_k)\succ
T(m_1,\ldots,m_k)$ if and only if $(m_1,\ldots,m_k)\succ
(n_1,\ldots,n_k)$, where $\sum_{i=1}^kn_i=\sum_{i=1}^km_i$.
\end{abstract}
\noindent{\small\noindent{\small {\it AMS Classification}}:
05C31, 05C69, 05C70.

\noindent\vspace{0.1mm}{\it Keywords}: Independence polynomial; Independent set; Largest root; Tree; Starlike tree; Partial order.}

\section{Introduction}

Throughout this paper we will consider only simple graphs. Let
$G=(V,E)$ be a simple graph. The {\it order} of $G$ denotes the
number of vertices of $G$. For every vertex $v\in V$, the {\it
closed neighborhood} of $v$ is the set $[v]=\{u \in V\,:\,uv\in E\}\cup
\{v\}$. For every edge $e\in E$ with end points $u$ and $v$, the {\it
closed neighborhood} of $e$ is the set $[e]=[u]\cup[v]$. For two graphs
$G_1=(V_1,E_1)$ and $G_2=(V_2,E_2)$, the {\it disjoint union} of
$G_1$ and $G_2$ denoted by $G_1+G_2$ is the graph with vertex
set $V_1\cup V_2$ and edge set $E_1\cup E_2$. The
graph $rG$ denotes the disjoint union of $r$ copies of
$G$.

A set $S\subseteq V(G)$ is an {\it independent set} if there is
no edge between the vertices of $S$. The {\it independence
number} of $G$, $\alpha(G)$, is the maximum cardinality of an
independent set of $G$. The {\it
independence polynomial} of $G$, $I(G,x)$, is defined as
$I(G,x)=\sum_{k=0}^{\alpha(G)} s(G,k) x^{k}$, where $s(G,k)$ is the
number of independent sets of $G$ of size $k$, and $s(G,0)=1$. This
polynomial was first introduced by Gutman and Harary in
\cite{gh}. For more details see \cite{gg,ggg,gh,lm}. One can see that $s(G,1)=n$ and $s(G,2)={{n} \choose 2}-m$, where $n$
and $m$ are the number of vertices and the number of edges of $G$, respectively. This shows that by independence polynomial one can obtain the number of vertices and the number of edges of the graph.

The roots of  independence polynomial like other graph polynomials
such as characteristic polynomial~\cite{m}, chromatic polynomial~\cite{bl}, domination polynomial~\cite{aaop}, edge cover polynomial~\cite{co} and  matching polynomial~\cite{f},
reflect some important information about the structure of graphs. Unlike characteristic polynomial and matching polynomial, independence polynomial of some graphs has non-real roots. It was conjectured that \cite{h}, for every claw-free graph, the
independence polynomial has only real roots. Recently,  M.
Chudnovsky and P. Seymour \cite{cs} showed that this conjecture
is valid. It was proved that the root of the smallest modulus of the independence polynomial of any
graph is real~\cite{bdn}. There are some graphs, for instance $K_{1,3}$, in which
their independence polynomial have non-real roots. The {\it
complete graph}, the {\it cycle}, and the {\it path} of order
$n$, are denoted  by $K_n$, $C_n$ and $P_n$, respectively. We
denote  the {\it complete bipartite graph} with part sizes $m$
and $n$, by $K_{m,n}$. Also $K_{1,n}$ is called a {\it star}. We let $S_n=K_{1,n-1}$. For
every vertex $v\in V(G)$, the {\it degree} of $v$ is the number
of edges incident with $v$ and is denoted by $deg_G(v)$. For
simplicity we write $deg(v)$ instead of $deg_G(v)$. By $\Delta(G)$ we mean the maximum degree of vertices of $G$.  A {\it starlike} tree is a tree which has only one
vertex of degree greater than two. Let $x_1,\ldots,x_n$ be some real numbers. By $X=\{x_1,\ldots,x_n\}$ we mean the multiset $X$ such that $x_1,\ldots,x_n$ are its member. In the other words in this paper set is multiset.
\\

It is well known that all roots of characteristic polynomial are real. Let $G$ and $H$ be two graphs. Let $\phi(G,\lambda)$ and $\phi(H,\lambda)$ be their characteristic polynomials, respectively. Suppose that $\Lambda(G)$ and $\Lambda(H)$ are the largest eigenvalue, largest root of the characteristic polynomial, of $G$ and $H$, respectively. L. Lov\'{a}sz and J. Pelik\'{a}n in  \cite{lovasz} defined $G\succ H$ if and only if $\phi(H,\lambda)\geq\phi(G,\lambda)$ for every $\lambda$ in the interval $[\Lambda(G),\infty)$.

Now, similarly we define an new ordering on the set of all simple graphs as follows:
Let $G$ and $H$ be two graphs. Let $I(G,x)$ and $I(H,x)$ be their independence polynomial, respectively. Since all coefficients of the independence polynomial are positive, all its real root are negative. Assume that $\xi(G)$ and $\xi(H)$ are the largest real roots of $I(G,x)$ and $I(H,x)$, respectively (In \cite{bdn} it was proved that the independence polynomial has at least one real root). We let $$G\succeq H,\, {\rm if\, and\, only\, if }\,\, I(H,x)\geq I(G,x),\,\, {\rm for\, every\,} x\in[\xi(G),0].$$
Also, we let $G\succ H$, if and only if $G\succeq H$ and $I(G,x)\neq I(H,x)$. We say
$G$ and $H$ are {\it${\cal I}$-equivalent},
%written $G\sim H$,
if $I(G,x)=I(H,x)$.
%The {\it${\cal I}$-equivalence class} of $G$ is defined as $[G]=\{H: H\sim G\}$. A
%graph $G$ is said to be {\it${\cal I}$-unique}, if $[G]=\{G\}$.
The structure of this paper is the following. In the next section we states some results about the $\xi(G)$. In section $3$ we investigate the ordering $\succeq$. In section $4$ we obtain the minimum and the maximum element of trees with respect to the ordering $\succeq$. In section $5$ we define some orderings on $\mathbb{R}^n$. Finally in the last section we investigate about the starlike trees and the ordering $\succeq$.

\section{Some properties of independence polynomial and its largest real root}

In this section we state some results about the independence polynomial and its largest real root. Let $G$ be a graph of order $n$.
We note that in some papers, the independence polynomial of $G$ is defined as $\sum_{k=0}^ns(G,k)(-x)^k$, where $s(G,k)$ is the number of independent sets of $G$ with cardinality $k$. We also mention that, for investigating about the independence polynomials, several authors consider the {\it clique
polynomial} of the graph $G$ which is $I(\overline{G}, x)$, where $\overline{G}$ is the complement of $G$.
By $\xi(G)$ we mean the largest real root of $I(G,x)$.
As we mentioned before, in \cite{bdn} it was shown that the independence polynomials have at least one real root.
Also in~\cite{pcs} and \cite{haji} it was proved by different ways.

\begin{lem}\label{product} {\rm \cite{gh}}
Let $G$ be a graph with connected components $G_{1},\ldots,G_{k}$.
Then $I(G,x)=\prod_{i=1}^{k}I(G_{i},x)$.
\end{lem}

\begin{lem}\label{vertex-edge} {\rm \cite{gh}}
Let $G$ be a graph. Then the following hold:
\begin{enumerate}
\item[1)]Let $v$ be a vertex of $G$. Then $I(G,x)=I(G\setminus v,x)+xI(G\setminus[v],x)$.
\item[2)]Let $e$ be an edge of $G$. Then $I(G,x)=I(G\setminus e,x)-x^2I(G\setminus[e],x)$.
\end{enumerate}
\end{lem}

%\begin{lem}\label{derivative} {\rm \cite{gh}}
%Let $G$ be a graph. Let $I'(G,x)$ be the derivative of $I(G,x)$ with respect to $x$. Then $I'(G,x)=\sum_{v\in V(G)}I(G\setminus [v],x)$.
%\end{lem}

\begin{thm}\label{real root} {\rm \cite{pcs}}
Let $G$ be a graph and $H$ be a subgraph of $G$. Then $I(G,x)$ and $I(H,x)$ have at least one real root. Moreover, $\xi(G)\geq \xi(H)$, where $\xi(G)$ and $\xi(H)$ are the largest real root of $I(G,x)$ and $I(H,x)$, respectively.
\end{thm}
%\begin{proof}{We proceed  by induction on $r(G)=|V(G)|+|E(G)|$. If $r(G)=1$, then $G=H=K_1$ and $\xi(G)=\xi(H)=-1$. Thus we are done. Now, let $r(G)\geq 2$.

%Let $v\in V(G)$. By the induction hypothesis $I(G\setminus v,x)$ has a real root. By the first part of Lemma~\ref{vertex-edge},
%$I(G,x)=I(G\setminus v,x)+xI(G\setminus[v],x)$. Putting $x_0=\xi(G\setminus v)$, we obtain that $I(G,x_0)=x_0I(G\setminus[v],x_0)$. Since $G\setminus[v]$ is a subgraph of $G\setminus v$,  by the induction hypothesis $\xi(G\setminus v)\geq \xi(G\setminus[v])$. On the other hand $I(G\setminus[v],0)=1$, so by mean value Theorem for continues functions, $I(G\setminus[v],x_0)\geq 0$. Therefore $I(G,x_0)=x_0I(G\setminus[v],x_0)\leq 0$. Since $I(G,0)=1$, by mean value Theorem, we conclude that $I(G,x)$ has a root in the interval $[x_0,0]$. This shows that $\xi(G)\geq \xi(G\setminus v)$.

%Similarly, for every edge $e$ of $G$, we obtain that $\xi(G)\geq \xi(G\setminus e)$. Now, let $H\neq G$. Thus there exists a vertex $v$ or an edge $e$ such that $H$ is a subgraph of $G\setminus v$ or $G\setminus e$. For instance suppose that $H$ is a subgraph of $G\setminus v$. Using induction hypothesis for $G\setminus v$, we get $\xi(G\setminus v)\geq \xi(H)$. This shows that $\xi(G)\geq\xi(H)$. Similarly, if $H$ is a subgraph
 %of $G\setminus e$, the result follows.}
%\end{proof}

\begin{thm}\label{connected} {\rm \cite{pcs}}
Let $G$ be a connected graph and $H$ be a subgraph of $G$. Then the following hold:
\begin{enumerate}
\item[1)]$\xi(G)\geq\xi(H)$, and the equality holds if and only if $G=H$.
\item[2)]The multiplicity of $\xi(G)$ is 1.
\end{enumerate}
\end{thm}

\section{A new ordering on the family of simple graphs}

By Theorem~\ref{real root}, the independence polynomial has a real root. Let {\it${\cal A}$ {\rm be the set of all simple graphs. We define an ordering $\succeq$ on} {\it${\cal A}$. {\rm Let $G$ and $H$ be two graphs. Let $\xi(G)$ and $\xi(H)$ be the largest real root of $I(G,x)$ and $I(H,x)$, respectively. We define $$G\succeq H,\, {\rm if\, and\, only\, if }\,\, I(H,x)\geq I(G,x),\,\, {\rm for\, every\,} x\in[\xi(G),0].$$
Also, we let $G\succ H$, if and only if $G\succeq H$ and $I(G,x)\neq I(H,x)$. In other words, $G\succ H$, if and only if $G\succeq H$ and $G,H$ have different independence polynomials. Note that $G\succ H$, implies that there exist $x_0\in[\xi(G),0)$ such that $I(H,x_0)>I(G,x_0)$.}

\begin{remark}\label{root greater} Let $G\succeq H$. Since for every $x\in(\xi(G),0]$, $I(G,x)>0$, thus $I(H,x)>0$ on the interval $(\xi(G),0]$. This shows that $\xi(G)\geq\xi(H)$.
\end{remark}

\begin{remark}\label{total order} $(\it{\cal A},\succeq)$ is not a total order set. Consider the graphs $K_2$ and $3K_1$.
We have $I(K_2,x)=1+2x$ and $I(3K_1,x)=(1+x)^3$. So $\xi(K_2)=-\frac{1}{2}$ and $\xi(3K_1)=-1$. One can see that the graphs $3K_1$ and $K_2$ are not comparable, that is $3K_1\nsucceq K_2$ and $K_2\nsucceq 3K_1$.
\end{remark}
{\rm The following result shows that} $(\it{\cal A},\succeq)$ {\rm is  a partially order set. Note that the antisymmetric  property holds up to} {\it${\cal I}$- equivalent}.
\begin{thm}\label{partial order} $(\it{\cal A},\succeq)$ is  a poset.
\end{thm}
\begin{proof}{Let $G,H$ and $K$ be some  graphs. Obviously, $G\succeq G$. Suppose that $G\succeq H$ and  $H\succeq G$. By Remark~\ref{root greater}, $\xi(G)=\xi(H)$. Thus $I(G,x)=I(H,x)$ for every $x\in [\xi(G),0]$. This shows that the polynomial $I(G,x)-I(H,x)$ has infinity number  roots. Thus $I(G,x)-I(H,x)\equiv 0$. So $I(G,x)=I(H,x)$ for every real number $x$, that is $G$ and $H$ are {\it${\cal I}$-equivalent}. Now, suppose that $G\succeq H$ and $H\succeq K$. Using Remark~\ref{root greater}, we conclude that $\xi(G)\geq\xi(H),\xi(K)$. This completes the proof.}
\end{proof}

\begin{remark}\label{cycle-star} There are some non-isomorphic {\it${\cal I}$-equivalent} graphs. For instance it is not hard to see that the cycle $C_n$ and $G_n$, where $G_n$ is the graph with the vertex set $\{1,\ldots,n\}$ and the edge set $\{12,23,34,\ldots,(n-1)n\}\cup\{(n-2)n\}$, have the same independence polynomial \cite{mro}. See Figure~\ref{Gn}.
\end{remark}
%$\{12,23,34,\ldots,\\(n-1)n\}\cup\{(n-2)n\}$
\begin{figure}[htb]
\centering
\begin{tikzpicture}[scale=1]
\filldraw [black]
(-3,0) circle (3.5 pt)
(-2,0) circle (3.5 pt)
(-1.5,0) circle (1 pt)
(-1,0) circle (1 pt)
(-0.5,0) circle (1 pt)
(0,0) circle (3.5 pt)
(1,0.5) circle (3.5 pt)
(1,-0.5) circle (3.5 pt);
\node [label=below:$G_n$] (G_n) at (-0.65,-0.75) {};
\node [label=above:$1$] (1) at (-3,0) {};
\node [label=above:$2$] (2) at (-2,0) {};
\node [label=above:$n-2$] (n-2) at (0,0) {};
\node [label=right:$n-1$] (n-1) at (1,0.5) {};
\node [label=right:$n$] (n) at (1,-0.5) {};
\draw[thick] (-3,0) --(-2,0);
\draw[thick] (-2,0) -- (-1.6,0);
\draw[thick] (-0.4,0) -- (0,0);
\draw[thick] (0,0) -- (1,0.5);
\draw[thick] (0,0) -- (1,-0.5);
\draw[thick] (1,0.5) -- (1,-0.5);
\end{tikzpicture}
\caption{$I(G_n,x)=I(C_n,x).$}\label{Gn}
\end{figure}
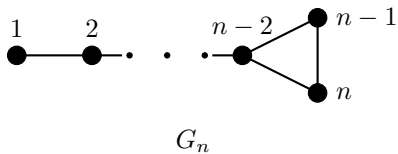
{\rm We guess that the set of all graphs (trees) with $n$ vertices is a total order set with respect to the ordering $\succeq$.}

\begin{conj}\label{total-tree} Let $T_1$ and $T_2$ be two trees of order $n$. Then $T_1\succeq T_2$ or $T_2\succeq T_1$.
\end{conj}

{\rm In the last section we prove Conjecture~\ref{total-tree} for starlike trees. In sequel we obtain some results about the ordering $\succeq$. One can easily  prove the following theorem.}

\begin{thm}\label{plus minus} Let $G_1,G_2,H_1,H_2$ and $H$ be some graphs. Then the following hold:
\begin{enumerate}
\item[1)] If $G_1\succeq G_2$  and $H_1\succeq H_2$, then $G_1+H_1\succeq G_2+H_2$.
\item[2)] If $G_1+H\succeq G_2+H$ and $\xi(G_1)\geq \xi(H)$, then $G_1\succeq G_2$.
\end{enumerate}
\end{thm}

\begin{thm}\label{subgraph} Let $G$ be a graph and $H$ be a proper subgraph of $G$. Then $G\succ H$.
\end{thm}
\begin{proof}{Since $H$  is a proper subgraph of $G$, $I(G,x)\neq I(H,x)$. Thus it remains to show that $G\succeq H$. Since $\succeq$ is a transitive relation, it suffices to show that for every vertex $v$ and edge $e$ of $G$, $G\succeq G\setminus v$ and $G\succeq G\setminus e$. First, we prove that $G\succeq G\setminus v$. Since $G\setminus[v]$ is a subgraph of $G$ by Theorem~\ref{real root}, $\xi(G)\geq\xi(G\setminus[v])$. Thus by the mean value Theorem for continuous functions, we conclude that $I(G\setminus[v],x)\geq 0$, on the interval $[\xi(G),0]$. Now, using the first part of Lemma~\ref{vertex-edge}, we obtain that $ I(G\setminus v,x)\geq I(G,x)$ on the interval $[\xi(G),0]$. Thus $G\succeq G\setminus v$. Similarly, one can prove that $G\succeq G\setminus e$.}
\end{proof}

\begin{thm}\label{two part} Let $G$ and $H$ be two graphs. Let $u\in V(G)$, $v\in V(H)$ and $e\in E(G)$, $e\,'\in E(H)$. Then the following hold:
\begin{enumerate}
\item[1)] If $G\setminus u\succeq H\setminus v$ and $H\setminus [v]\succeq G\setminus [u]$, then $G\succeq H$.
\item[2)] If $G\setminus e\succeq H\setminus e\,'$ and $H\setminus [e\,']\succeq G\setminus [e]$, then $G\succeq H$.
\end{enumerate}
\end{thm}
\begin{proof}{
\

\begin{enumerate}
\item[1)] By Theorem~\ref{subgraph}, $G\succeq G\setminus u$. On the other hand, $G\setminus u\succeq H\setminus v$.
 Thus by Theorem~\ref{real root}, $\xi(G)\geq\xi(G\setminus u)\geq\xi(H\setminus v)$. So
 \begin{equation}\label{1-two part}
\hbox{$I(H\setminus v,x)\geq I(G\setminus u,x)$ on the interval $[\xi(G),0]$.}
\end{equation}
Similarly, we obtain $\xi(H)\geq\xi(H\setminus [v])\geq\xi(G\setminus [u])$ and
\begin{equation}\label{2-two part}
\hbox{$I(G\setminus[u],x)\geq I(H\setminus[v],x)$ on the interval $[\xi(G),0]$.}
\end{equation}
Now, the proof follows from the equations~\ref{1-two part}, \ref{2-two part} and using the first  part of Lemma~\ref{vertex-edge}, for the graphs $G,H$ and the vertices $u,v$.
\item[2)] Similar to the previous part, one can prove that $G\succeq H$.
\end{enumerate}
}
\end{proof}

{\rm Using Theorem~\ref{two part}, One can prove the following theorem.}

\begin{thm}\label{two part strict} Let $G$ and $H$ be two graphs. Let $u\in V(G)$, $v\in V(H)$ and $e\in E(G)$, $e\,'\in E(H)$. Then the following hold:
\begin{enumerate}
\item[1)] If $G\setminus u\succeq H\setminus v$ and $H\setminus [v]\succ G\setminus [u]$, or $G\setminus u\succ H\setminus v$ and
 $H\setminus [v]\succeq G\setminus [u]$, then $G\succ H$.
\item[2)] If $G\setminus e\succeq H\setminus e\,'$ and $H\setminus [e\,']\succ G\setminus [e]$, or $G\setminus e\succ H\setminus e\,'$ and $H\setminus [e\,']\succeq G\setminus [e]$, then $G\succ H$.
\end{enumerate}
\end{thm}

\section{The minimal and the maximal element of the family of trees with respect to the ordering $\succeq$}

{\rm In this section we investigate the minimal and the maximal element of the family of all trees with respect to the ordering $\succeq$. We show that for every tree $T$ of order $n$, $S_n\succeq T\succeq P_n$, where $S_n$ and $P_n$ are the star and the path of order $n$, respectively. First we introduce an operation on graphs. Let $G$ be a graph. We define an operation $\star$ on $G$ as follows: Let $u$ be a vertex of $G$ with degree $1$. Suppose $v$ is a vertex of $G$ with degree at least $3$ such that it has the shortest distance from $u$ among all vertices of degree at least $3$. Let $w$ be a vertex of $G$ adjacent to $v$. By $G^{ \star}_{u,v,w}$ we mean the graph $G\setminus vw+uw$. See Figure~\ref{star star}. Note that the order and the size of $G$ and $G^{ \star}_{u,v,w}$ are the same. Also, the number of vertices of degree $1$ in $G^{ \star}_{u,v,w}$ is one less from the number of vertices of degree $1$ in $G$. For example let $G_n$ be the graph with the vertex set $\{1,\ldots,n\}$ and the edge set $\{12,23,34,\ldots,(n-1)n\}\cup\{(n-2)n\}$ (see Figure~\ref{Gn}). Then $({G^{\star}_n})_{1,n-2,n}$ is the cycle $C_n$.}
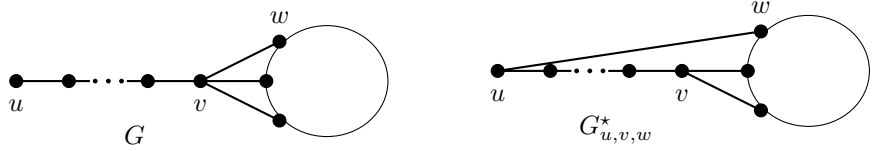
\begin{figure}[htb]
\centering
\begin{tikzpicture}[scale=0.7]
\filldraw [black]
(-2,0) circle (3.5 pt)
(-1,0) circle (3.5 pt)
(-0.5,0) circle (1 pt)
(-0.25,0) circle (1 pt)
(0,0) circle (1 pt)
(0.5,0) circle (3.5 pt)
(1.5,0) circle (3.5 pt)
(2.75,0) circle (3.5 pt)
(3,0.75) circle (3.5 pt)
(3,-0.75) circle (3.5 pt);
\node [label=below:$G$] (G_{u,v,w}) at (0.25,-0.5) {};
\node [label=below:$u$] (u) at (-2,0) {};
\node [label=below:$v$] (v) at (1.5,0) {};
\node [label=above:$w$] (w) at (3,0.75) {};
\draw[thick] (-2,0) -- (-1,0);
\draw[thick] (-1,0) -- (-0.60,0);
\draw[thick] (0.10,0) -- (0.5,0);
\draw[thick] (0.5,0) -- (1.5,0);
\draw[thick] (1.5,0) -- (3,0.75);
\draw[thick] (1.5,0) -- (3,-0.75);
\draw[thick] (1.5,0) -- (2.75,0);
\draw (3.90,0) ellipse (33pt and 30pt);
\end{tikzpicture}
\hspace{1cm}
\begin{tikzpicture}[scale=0.7]
\filldraw [black]
(-2,0) circle (3.5 pt)
(-1,0) circle (3.5 pt)
(-0.5,0) circle (1 pt)
(-0.25,0) circle (1 pt)
(0,0) circle (1 pt)
(0.5,0) circle (3.5 pt)
(1.5,0) circle (3.5 pt)
(2.75,0) circle (3.5 pt)
(3,0.75) circle (3.5 pt)
(3,-0.75) circle (3.5 pt);
\node [label=below:$G_{u,v,w}^{\star}$] (G_{u,v,w}^*) at (0.25,-0.5) {};
\node [label=below:$u$] (u) at (-2,0) {};
\node [label=below:$v$] (v) at (1.5,0) {};
\node [label=above:$w$] (w) at (3,0.75) {};
\draw[thick] (-2,0) -- (-1,0);
\draw[thick] (-1,0) -- (-0.60,0);
\draw[thick] (0.10,0) -- (0.5,0);
\draw[thick] (0.5,0) -- (1.5,0);
\draw[thick] (-2,0) -- (3,0.75);
\draw[thick] (1.5,0) -- (3,-0.75);
\draw[thick] (1.5,0) -- (2.75,0);
\draw (3.90,0) ellipse (33pt and 30pt);
\end{tikzpicture}
\caption{The operation $\star$.}\label{star star}
\end{figure}

\begin{thm}\label{star operation} Let $G$, $u,v,w$ and  $G^{ \star}_{u,v,w}$ be as mentioned above. Then $G\succeq G^{ \star}_{u,v,w}$.
\end{thm}
\begin{proof}{ Since $G\setminus vw=G^{ \star}_{u,v,w}\setminus uw$,  $G\setminus vw\succeq G^{ \star}_{u,v,w}\setminus uw$. On the other hand  $G\setminus [vw]$  is an induced subgraph of $G^{ \star}_{u,v,w}\setminus [uw]$. Thus by Theorem~\ref{subgraph}, $G^{ \star}_{u,v,w}\setminus [uw]\succeq G\setminus [vw]$. The second part of Theorem~\ref{two part}, completes the proof.}
\end{proof}

{\rm For $i=1,\ldots,k$, let $n_i\geq 2$ be a natural number. By $T=T(n_1,\ldots,n_k)$ we mean the tree $T$ which has a
vertex $v$ of degree $k$ such that $T\setminus
v=P_{n_1-1}+\cdots+P_{n_k-1}$ (see Figure~\ref{Tnk}). Clearly, the order of $T(n_1,\ldots,n_k)$ is $n_1+\cdots+n_k-k+1$.
%We let $V(P_{n_i-1})=\{v^i_1,\ldots,v^i_{n_i-1}\}$ and $E(P_{n_i-1})=\{v^i_1v^i_2,v^i_2v^i_3,\ldots,v^i_{n_i-2}v^i_{n_i-1}\}$. So $V(T(n_1,\ldots,n_k))=\bigcup_{i=1}^{k}V(P_{n_i-1})\bigcup\{v\}$ and $E(T(n_1,\ldots,n_k))=\bigcup_{i=1}^{k}E(P_{n_i-1})\bigcup\{vv^i_1\}$.
Note that $T(n_1,\ldots,n_k)$ is a path or a {\it starlike} tree (a tree which has exactly one vertex of degree greater than two). Let $T_{n,k}=T(n-k+1,\underbrace{2,\ldots,2}_{k-1})$ (see Figure~\ref{Tnk}). In particular $T_{n,1}=T_{n,2}=P_n$. By $H_{n,k}$ we mean the tree that is shown in Figure~\ref{Hnk}. We have $H_{n,k}=H_{n,n-k}$ and $H_{n,1}=S_n$.}
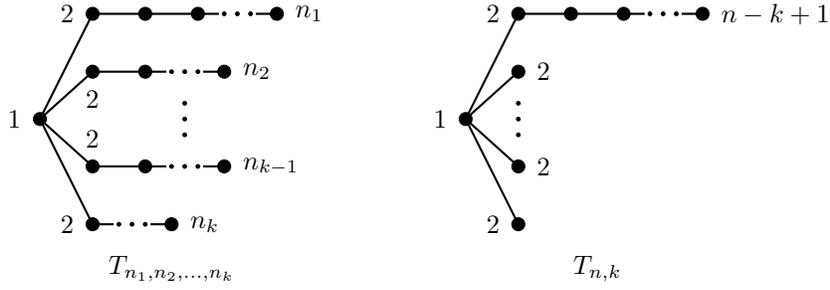
\begin{figure}[htb]
\centering
\begin{tikzpicture}[scale=0.7]
\filldraw [black]
(0,0) circle (3.5 pt)
(1,2) circle (3.5 pt)
(1,0.9) circle (3.5 pt)
(1,-0.9) circle (3.5 pt)
(1,-2) circle (3.5 pt)
(2,2) circle (3.5 pt)
(2,0.9) circle (3.5 pt)
(2,-0.9) circle (3.5 pt)
(3,2) circle (3.5 pt)
(3.5,2) circle (1 pt)
(3.75,2) circle (1 pt)
(4,2) circle (1 pt)
(4.5,2) circle (3.5 pt)
(2.5,0.9) circle (1 pt)
(2.75,0.9) circle (1 pt)
(3,0.9) circle (1 pt)
(3.5,0.9) circle (3.5 pt)
(2.5,-0.9) circle (1 pt)
(2.75,-0.9) circle (1 pt)
(3,-0.9) circle (1 pt)
(3.5,-0.9) circle (3.5 pt)
(1.5,-2) circle (1 pt)
(1.75,-2) circle (1 pt)
(2,-2) circle (1 pt)
(2.5,-2) circle (3.5 pt)
(2.75,0.30) circle (1 pt)
(2.75,0) circle (1 pt)
(2.75,-0.30) circle (1 pt);
\node [label=below:$T_{n_1,n_2,...,n_k}$] (T_{n_1,n_2,...,n_k}) at (2.5,-2.25) {};
\node [label=left:$1$] (1) at (0,0) {};
\node [label=left:$2$] (2) at (1,2) {};
\node [label=below:$2$] (2) at (1,0.9) {};
\node [label=above:$2$] (2) at (1,-0.9) {};
\node [label=left:$2$] (2) at (1,-2) {};
\node [label=right:$n_1$] (n_1) at (4.5,2) {};
\node [label=right:$n_2$] (n_2) at (3.5,0.9) {};
\node [label=right:$n_{k-1}$] (n_{k-1}) at (3.5,-0.9) {};
\node [label=right:$n_{k}$] (n_{k}) at (2.5,-2) {};
\draw[thick] (0,0) -- (1,2);
\draw[thick] (0,0) -- (1,0.9);
\draw[thick] (0,0) -- (1,-0.9);
\draw[thick] (0,0) -- (1,-2);
\draw[thick] (1,2) -- (2,2);
\draw[thick] (1,0.9) -- (2,0.9);
\draw[thick] (1,-0.9) -- (2,-0.9);
\draw[thick] (2,2) -- (3,2);
\draw[thick] (3,2) -- (3.40,2);
\draw[thick] (4.10,2) -- (4.5,2);
\draw[thick] (2,0.9) -- (2.40,0.9);
\draw[thick] (3.10,0.9) -- (3.5,0.9);
\draw[thick] (2,-0.9) -- (2.40,-0.9);
\draw[thick] (3.10,-0.9) -- (3.5,-0.9);
\draw[thick] (1,-2) -- (1.40,-2);
\draw[thick] (2.10,-2) -- (2.5,-2);
\end{tikzpicture}
\hspace{1cm}
\begin{tikzpicture}[scale=0.7]
\filldraw [black]
(0,0) circle (3.5 pt)
(1,2) circle (3.5 pt)
(1,0.9) circle (3.5 pt)
(1,-0.9) circle (3.5 pt)
(1,-2) circle (3.5 pt)
(2,2) circle (3.5 pt)
(3,2) circle (3.5 pt)
(3.5,2) circle (1 pt)
(3.75,2) circle (1 pt)
(4,2) circle (1 pt)
(4.5,2) circle (3.5 pt)
(1,0.30) circle (1 pt)
(1,0) circle (1 pt)
(1,-0.30) circle (1 pt);
\node [label=below:$T_{n,k}$] (T_{n,k}) at (2.5,-2.25) {};
\node [label=left:$1$] (1) at (0,0) {};
\node [label=left:$2$] (2) at (1,2) {};
\node [label=right:$2$] (2) at (1,0.9) {};
\node [label=right:$2$] (2) at (1,-0.9) {};
\node [label=left:$2$] (2) at (1,-2) {};
\node [label=right:$n-k+1$] (n-k+1) at (4.5,2) {};
\draw[thick] (0,0) -- (1,2);
\draw[thick] (0,0) -- (1,0.9);
\draw[thick] (0,0) -- (1,-0.9);
\draw[thick] (0,0) -- (1,-2);
\draw[thick] (1,2) -- (2,2);
\draw[thick] (2,2) -- (3,2);
\draw[thick] (3,2) -- (3.40,2);
\draw[thick] (4.10,2) -- (4.5,2);
\end{tikzpicture}
\caption{The trees $T(n_1,\ldots,n_k)$ and $T_{n,k}$.}\label{Tnk}
\end{figure}
\vspace{.1cm}
\begin{figure}[htb]
\centering
\begin{tikzpicture}[scale=.8]
\filldraw [black]
(-1,0) circle (3.5 pt)
(1,0) circle (3.5 pt)
(-2,1) circle (3.5 pt)
(-2,0.5) circle (3.5 pt)
(-2,-1) circle (3.5 pt)
(-2,0.1) circle (1 pt)
(-2,-0.3) circle (1 pt)
(-2,-0.7) circle (1 pt)
(2,0.1) circle (1 pt)
(2,-0.3) circle (1 pt)
(2,-0.7) circle (1 pt)
(2,1) circle (3.5 pt)
(2,0.5) circle (3.5 pt)
(2,-1) circle (3.5 pt);
\node [label=below:$H_{n,k}$] (H_{n,k}) at (0,-0.75) {};
\node [label=left:$1$] (1) at (-2,1) {};
\node [label=left:$2$] (2) at (-2,0.5) {};
\node [label=left:$k-1$] (k-1) at (-2,-1) {};
\node [label=right:$1$] (1) at (2,1) {};
\node [label=right:$2$] (2) at (2,0.5) {};
\node [label=right:$n-k-1$] (n-k-1) at (2,-1) {};
\draw[thick] (-1,0) -- (1,0);
\draw[thick] (-1,0) -- (-2,1);
\draw[thick] (-1,0) -- (-2,0.5);
\draw[thick] (-1,0) -- (-2,-1);
\draw[thick] (1,0) -- (2,1);
\draw[thick] (1,0) -- (2,0.5);
\draw[thick] (1,0) -- (2,-1);
\end{tikzpicture}
\caption{The tree $H_{n,k}$.}\label{Hnk}
\end{figure}
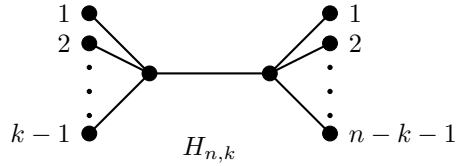
\begin{figure}[htb]
\centering
\begin{tikzpicture}[scale=0.7]
\filldraw [black]
(0,0) circle (3.5 pt)
(0,-2) circle (3.5 pt)
(2,-2) circle (3.5 pt)
(-2,-2) circle (3.5 pt)
(.075,-2) circle (1 pt)
(1,-2) circle (1 pt)
(1.25,-2) circle (1 pt);
\node [label=below:$T$] (T) at (0,-5.5) {};
\node [label=left:$v$] (v) at (0.25,0.25) {};
\node [label=left:$v_k$] (v_k) at (2.75,-1.75) {};
\node [label=above:$v_2$] (v_2) at (-.25,-2.1) {};
\node [label=left:$v_1$] (v_1) at (-2,-2) {};
\draw[thick] (0,0) -- (-2,-2);
\draw[thick] (0,0) -- (0,-2);
\draw[thick] (0,0) -- (2,-2);
\draw (-1.95,-3.35) ellipse (20pt and 40pt);
\node [label=below:$H_1$] (H_1) at (-2,-2.75) {};
\draw (0,-3.35) ellipse (20pt and 40pt);
\node [label=below:$H_2$] (H_2) at (0,-2.75) {};
\draw (2.05,-3.35) ellipse (20pt and 40pt);
\node [label=below:$H_k$] (H_k) at (2,-2.75) {};
\end{tikzpicture}
\caption{the tree $T$.}\label{T}
\end{figure}
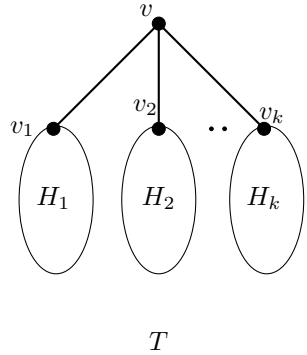
{\rm Now, we are in a position to determine  the minimal and the maximal element of trees with respect to the ordering $\succeq$.}
\begin{thm}\label{TH} Let $T$ be a tree of order $n$ with maximum degree $k$. Then $$H_{n,k}\succeq T\succeq T_{n,k}.$$
Moreover, in the left hand side the equality holds if and only if $T=H_{n,k}$, and in the right hand side the equality holds if and only if $T=T_{n,k}$.
\end{thm}
\begin{proof}{First we prove that $T\succeq T_{n,k}$.  Let $v$ be a vertex of $T$ with degree $k$. Let $N(v)=\{v_1,\ldots,v_k\}$ be the set of all neighbors of $v$. Suppose that $T\setminus v=H_1+\cdots+H_k$, such that $v_i\in V(H_i)$ (see Figure~\ref{T}). Let $T_i=H_i+vv_i$ and $n_i$ be the order of $T_i$, for $i=1,\ldots,k$.

 We claim that $T\succeq T(n_1,\ldots,n_k)$. If $T_i=P_{n_i}$, for $i=1,\ldots,k$, then $T=T(n_1,\ldots,n_k)$. So in this case the claim is proved. Now, suppose that $T\neq T(n_1,\ldots,n_k)$. Without loss of generality we may assume that $T_k\neq P_{n_k}$. Let $u\neq v$ be a vertex of $T_k$ with degree $1$. Since $T_k$ is not a path, there exists a vertex $w\in V(T_k)$ of degree at least $3$ having the shortest distance from $u$ among all vertices of degree at least $3$. Let $z\neq v$ be a vertex of $T_k$ adjacent to $w$. Considering the trees $({T_k})^{\star}_{u,w,z}$ and $T^{\star}_{u,w,z}$, we conclude that $T^{\star}_{u,w,z}\setminus v=H_1+\cdots+H_{k-1}+({T_k})^{\star}_{u,w,z}\setminus v$. Since the operation $\star$ reduces the number of vertices of degree $1$, using the operation $\star$ ( applying this operation $(t-k)$ times, where $t$ is the number of vertices of $T$ with degree $1$), one obtains the tree $T(n_1,\ldots,n_k)$. Using Theorem~\ref{star operation}, we get $T\succeq T(n_1,\ldots,n_k)$. Thus the claim is proved.

To complete the proof of the first part, it is sufficient to show that $T(n_1,\ldots,n_k)\succeq T_{n,k}$. Let $T_1=T(n_1,\ldots,n_k)$.
Assume that $T_1\neq T_{n,k}$. We prove that $T_1\succ T_{n,k}$. Since $T_1\neq T_{n,k}$, there exist $i$ and $j$ such that $n_i,n_j\geq 3$. Without loss of generality suppose that $n_1,n_2\geq 3$. Let $v\in V(T_1)$ be a vertex of degree $k$. Thus $T_1\setminus
v=P_{n_1-1}+\cdots+P_{n_k-1}$. For $i=1,\ldots,k$, let $V(P_{n_i-1})=\{v^i_1,\ldots,v^i_{n_i-1}\}$ and $E(P_{n_i-1})=\{v^i_1v^i_2,v^i_2v^i_3,\ldots,v^i_{n_i-2}v^i_{n_i-1}\}$. Note that $V(T_1)=\bigcup_{i=1}^{k}V(P_{n_i-1})\bigcup\{v\}$ and $E(T_1)=\bigcup_{i=1}^{k}E(P_{n_i-1})\bigcup\{vv^i_1\}$. Let $T_2=T_1\setminus v^1_1v^1_2+v^1_{n_1-1}v^2_{n_2-1} $. Thus $T_2=T(n_1+n_2-2,2,n_3,\ldots,n_k)$.  Using this method several times one can obtain the tree $T_{n,k}$. Therefore, if $T(n_1,\ldots,n_k)\succ T(n_1+n_2-2,2,n_3,\ldots,n_k)$, equivalently $T_1\succ T_2$, the proof is complete. Now, we show that $T_1\succ T_2$. First note that $T_1\setminus v_1^1v_2^1=T_2\setminus v_{n_1-1}^1v_{n_2-1}^2$. We have the following cases. All of these cases are proved similarly. For example we prove the first case.
\begin{enumerate}
\item[1)] $n_1=n_2=3$. Then $T_1\setminus[v_1^1v_2^1]=P_2+P_{n_3-1}+\cdots+P_{n_k-1}$
and $T_2\setminus[v_{n_1-1}^1v_{n_2-1}^2]=T(2,n_3,\ldots,n_k)$. Thus $T_1\setminus[v_1^1v_2^1]$ is a proper subgraph of $T_2\setminus[v_{n_1-1}^1v_{n_2-1}^2]$. Using Theorem~\ref{subgraph}, we have $T_2\setminus[v_{n_1-1}^1v_{n_2-1}^2]\succ T_1\setminus[v_1^1v_2^1]$. On the other hand $T_1\setminus v_1^1v_2^1\succeq T_2\setminus v_{n_1-1}^1v_{n_2-1}^2$. By the second part of Theorem~\ref{two part strict}, we obtain $T_1\succ T_2$.
\item[2)] $n_1=3$ and $n_2\geq 4$.
\item[3)] $n_1\geq 4$ and $n_2=3$.
\item[4)]$n_1,n_2\geq 4$.
\end{enumerate}

Now, we prove the left hand side inequality. Let $T\neq H_{n,k}$. By induction on $n$ we show that $H_{n,k}\succ T$. It is easy to see that $n\geq 6$. If $n=6$, then $k=3$ and $T$ is $T(2,2,4)$ or $T(2,3,3)$. Since $I(H_{6,3},x)=1+6x+10x^2+6x^3+x^4$, $I(T(2,2,4),x)=1+6x+10x^2+3x^3+x^4$ and
$I(T(2,3,3),x)=1+6x+10x^2+5x^3$, one can see that $H_{6,3}\succ T(2,2,4)$ and $H_{6,3}\succ T(2,3,3)$. Suppose that $n\geq 7$. Let $u$ be a vertex of $T$ with degree $k$ and $N(u)=\{u_1,\ldots,u_k\}$ be the set of all neighbors of $u$. Suppose that $T\setminus u=T_1+\cdots+T_k$, such that $u_i\in V(T_i)$. Let $n_i$ be the order of $T_i$, for $i=1,\ldots,k$. If $n_1=\cdots=n_k=1$, then $k=n-1$ and $T=H_{n,n-1}$, a contradiction. So there exists $j$ such that $n_j\geq 2$. Let $w\neq u_j$ be a vertex of $T_j$ with degree $1$. Let $z$ be the neighbor of $w$. Let $z'$ be the vertex of $H_{n,k}$ that has degree $n-k$ and $w'$ be a neighbor of $z'$ with degree $1$. We have $T\setminus zw=K_1+T_0$, where $T_0$ is a tree of order $n-1$ and the maximum degree $k$. On the other hand $H_{n,k}\setminus z'w'=K_1+H_{n-1,k}$. By the induction hypothesis, $H_{n-1,k}\succeq T_0$. Using the first part of Theorem~\ref{plus minus}, we conclude that $H_{n,k}\setminus z'w'\succeq T\setminus zw$. We have $H_{n,k}\setminus[z'w']=(k-1)K_1$. On the other hand, the order of $T\setminus[zw]$ is at least $k-1$. If the order of $T\setminus[zw]$ is  $k-1$, then $T=H_{n,k}$, a contradiction.
Thus $H_{n,k}\setminus[z'w']$ is a proper subgraph of $T\setminus[zw]$. By Theorem~\ref{subgraph}, $T\setminus[zw]\succ H_{n,k}\setminus[z'w']$. Using the second part of Theorem~\ref{two part strict}, we conclude that $H_{n,k}\succ T$.}
\end{proof}
{\rm Similar to Theorem~\ref{TH}, by deleting some suitable edges, one can prove the following theorem. The next result states the relation between the trees $T_{n,k}$ and $H_{n,k}$, for $k=2,3,\ldots$.}

\begin{thm}\label{THchain} Let $n\geq 2$. Then the following hold:
\begin{enumerate}
\item[1)] $T_{n,n-1}\succ T_{n,n-2}\succ \cdots\succ T_{n,2}$.
\item[2)] $H_{n,n-1}\succ H_{n,n-2}\succ \cdots \succ H_{n,\lceil\frac{n}{2}\rceil}$.
\end{enumerate}
\end{thm}

\begin{cor}\label{star path} Let $T$ be a tree of order $n$. Then $S_n\succeq T\succeq P_n$. Moreover, in the left hand side the equality holds if and only if $T=S_n$ while in the right hand side the equality holds if and only if $T=P_n$.
\end{cor}
\begin{proof}{Let $\Delta(T)=k$. For $k\in\{1,2,n-1\}$, there is nothing to prove. Suppose that  $n-1>k>2$. By Theorems~\ref{TH} and \ref{THchain}, $S_n=H_{n,n-1}\succ H_{n,k}\succeq T\succeq T_{n,k}\succ T_{n,2}=P_n$.}
\end{proof}

\begin{remark}\label{length chain} Let $(A,\geq)$ be a poset. We recall that the {\it length} of the chain $a_1> a_2>\cdots>a_{\l}$ in $A$ is defined as $\l$. The first part of Theorem~\ref{THchain} shows that, there is a chain of length $n-2$ in the poset $({\cal T}_n,\succeq)$, where ${\cal T}_n$ is the set of all trees of order $n$.  We think that the length of any chain in $({\cal T}_n,\succeq)$ is at most $n-2$. We note that there are some trees having the same maximum degree while they are comparable. For example, in ${\cal T}_7$, one can see that $T(3,3,3)\succ T(2,2,5)$.
\end{remark}

\begin{conj}\label{length chain conj} Let ${\cal T}_n$ be the set of all trees of order $n$. Then the length of any chain in $({\cal T}_n,\succeq)$ is at most $n-2$.
\end{conj}

\section{Some properties of the poset $(\mathbb{R}^n,\succeq)$ and convertibility}

{\rm In this section we obtain some results related the ordering $\succeq$. We will use these results, in the next section, to investigate about Conjecture~\ref{total-tree} and show that this conjecture is valid for some families of starlike trees.

Let ${\cal A}_n=\{(x_1,\ldots,x_n)\in\mathbb{R}^n:\,x_1\geq x_2\geq\cdots\geq x_n\}$. Let $X=(x_1,\ldots,x_n),Y=(y_1,\ldots,y_n)\in{\cal A}_n$.  By $X\succeq Y$ we mean $X=Y$ or there exists $1\leq j\leq n$ such that $x_1=y_1,\ldots,x_{j-1}=y_{j-1}$ and $x_j>y_j$. We let $X\succ Y$ if and only if $X\succeq Y$ and $X\neq Y$. It is easy  to see that $({\cal A}_n,\succeq)$ is a totaly ordered set.  By $X\succeq_d Y$ we mean
$$X\succeq_d Y \Longleftrightarrow \sum_{i=1}^jx_i\geq \sum_{i=1}^jy_i, \hbox{ for}\,\, j=1,\ldots,n.$$
Also we let $X\succ_d Y$ if and only if $X\succeq_dY$ and $X\neq Y$.}

\begin{thm}\label{dominate} Let $X,Y\in {\cal A}_n$. If $X\succeq_dY$, then $X\succeq Y$.
\end{thm}
\begin{proof}{Suppose that $Y\succ X$. Thus there exists $j\in\{1,\ldots,n\}$ such that $x_1=y_1,\ldots,x_{j-1}=y_{j-1}$ and $y_j>x_j$. This shows that $\sum_{i=1}^jy_j>\sum_{i=1}^jx_j$, a contradiction.}
\end{proof}

\begin{remark}\label{dominate-converse} Note that the converse of Theorem~\ref{dominate} is not valid. Clearly $(7,2,2)\succeq(5,5,1)$ but $(7,2,2)\nsucceq_d(5,5,1)$.
\end{remark}

{\rm Let $n_1,m_1,\ldots,n_k,m_k$ be some real numbers. By $\{n_1,\ldots,n_k\}\succeq \{m_1,\ldots,m_k\}$ we mean there exist two permutations $\pi$ and $\sigma$ on the set $\{1,\ldots,k\}$ such that $n_{\pi(1)}\geq\cdots\geq n_{\pi(k)}$, $m_{\sigma(1)}\geq\cdots\geq m_{\sigma(k)}$ and $(n_{\pi(1)},\ldots,n_{\pi(k)})\succeq(m_{\sigma(1)},\ldots,m_{\sigma(k)})$. Also we let $\{n_1,\ldots,n_k\}\succeq_d \{m_1,\ldots,m_k\}$ if $(n_{\pi(1)},\ldots,n_{\pi(k)})\succeq_d(m_{\sigma(1)},\ldots,m_{\sigma(k)})$. Similarly one can define the notations $\{n_1,\ldots,n_k\}\succ \{m_1,\ldots,m_k\}$
and $\{n_1,\ldots,n_k\}\succ_d \{m_1,\ldots,m_k\}$.}

\begin{thm}\label{adding} Let $n_1,m_1,\ldots,n_k,m_k$ and $x_1,\ldots,x_t$ be some real numbers. Then the following hold:
\begin{enumerate}
\item[1)] $\{n_1,\ldots,n_k\}\succeq \{m_1,\ldots,m_k\}$ if and only if $$\{n_1,\ldots,n_k,x_1,\ldots,x_t\}\succeq\{m_1,\ldots,m_k,x_1,\ldots,x_t\}.$$
\item[2)] $\{n_1,\ldots,n_k\}\succeq_d \{m_1,\ldots,m_k\}$ if and only if
$$\{n_1,\ldots,n_k,x_1,\ldots,x_t\}\succeq_d\{m_1,\ldots,m_k,x_1,\ldots,x_t\}.$$
\end{enumerate}
\end{thm}
\begin{proof}{It suffices to prove theorem for $t=1$. On the other hand we can suppose that $n_1\geq\cdots\geq n_k$ and $m_1\geq\cdots\geq m_k$ (by changing the indexes). If $(n_1,\ldots,n_k)=(m_1,\ldots,m_k)$, there is nothing to prove. Now assume that there exists $j\in\{1,\ldots,k\}$ such that $n_1=m_1,\ldots,n_{j-1}=m_{j-1}$ and $n_j>m_j$. Let $\{n_1,\ldots,n_k,x_1\}=\{z_1,\ldots,z_{k+1}\}$ and $\{m_1,\ldots,m_k,x_1\}=\{w_1,\ldots,w_{k+1}\}$, such that $z_1\geq \cdots\geq z_{k+1}$ and $w_1\geq \cdots\geq w_{k+1}$. First, we prove that $(z_1,\ldots,z_{k+1})\succeq (w_1,\ldots,w_{k+1})$. We have the following cases:
\begin{enumerate}
\item[1)] $x_1\geq n_j$. Thus $z_1=w_1,\ldots ,z_j=w_j$ and $z_{j+1}=n_j$, $w_{j+1}=m_j$. We are done.
\item[2)] $n_j>x_1\geq m_j$. Therefore $z_1=w_1=n_1,\ldots,z_{j-1}=w_{j-1}=n_{j-1}$ and $z_j=n_j$, $w_j=x_1$. We are done.
\item[3)] $m_j>x_1$. So $z_1=w_1=n_1,\ldots,z_{j-1}=w_{j-1}=n_{j-1}$ and $z_j=n_j$, $w_j=m_j$. We are done.
\end{enumerate}

Conversely, suppose that $\{n_1,\ldots,n_k,x_1\}\succeq \{m_1,\ldots,m_k,x_1\}$. We show that $\{n_1,\ldots,n_k\}\succeq\{m_1,\ldots,m_k\}$. If $\{n_1,\ldots,n_k\}=\{m_1,\ldots,m_k\}$, then we are done. Let $\{n_1,\ldots,n_k\}\neq\{m_1,\ldots,m_k\}$. By  contradiction suppose that $\{m_1,\ldots,m_k\}\succeq \{n_1,\ldots,n_k\}$. By the first part of the theorem, $\{m_1,\ldots,m_k,x_1\}\succeq \{n_1,\ldots,n_k,x_1\}$. This shows that $\{n_1,\ldots,n_k\}=\{m_1,\ldots,m_k\}$, a contradiction.

Now, we prove the second part of the theorem. More precisely, we show that $(z_1,\ldots,z_{k+1})\succeq_d (w_1,\ldots,w_{k+1})$. We consider the following cases:
\begin{enumerate}
\item[1)] $x_1\geq n_j$. Thus $z_1=w_1,\ldots ,z_j=w_j$ and $z_1,\ldots,z_j\in\{n_1,\ldots,n_{j-1},x_1\}$.
   Clearly, $z_1+\cdots+z_i=w_1+\cdots+w_i$, for $i\in\{1,\ldots,j\}$.  Also for $i\geq j+1$, $\sum_{i=j+1}^hz_i=\sum_{i=j}^{h-1}n_i$ and $\sum_{i=j+1}^hw_i=\sum_{i=j}^{h-1}m_i$. Since $(n_1,\ldots,n_k)\succeq_d(m_1,\ldots,m_k)$, the proof is complete.
\item[2)] $n_j>x_1\geq m_j$. Therefore $z_1=w_1=n_1,\ldots,z_{j-1}=w_{j-1}=n_{j-1}$ and $z_j=n_j$, $w_j=x_1$. Let $n_t>x_1\geq n_{t+1}$, for some $t\geq j$. Considering the cases $j\geq h\geq 1$, $t\geq h\geq j+1$ and $h\geq t+1$, it is not hard to see that $\sum_{i=1}^hz_i\geq \sum_{i=1}^hw_i$. This shows that $(z_1,\ldots,z_{k+1})\succeq_d(w_1,\ldots,w_{k+1})$.
\item[3)] $m_j>x_1$. Similar to the previous cases one can obtain the result.
\end{enumerate}
Conversely, let $\{n_1,\ldots,n_k,x_1\}\succeq_d \{m_1,\ldots,m_k,x_1\}$. We show that $\{n_1,\ldots,n_k\}\succeq_d \{m_1,\ldots,m_k\}$. Suppose that $n_1\geq \cdots\geq n_r\geq x_1\geq n_{r+1}\geq\cdots\geq n_k$ and $m_1\geq \cdots\geq m_s\geq x_1\geq m_{s+1}\geq\cdots\geq m_k$, for some $r,s\in\{1,\ldots,k\}$. Considering the cases $r<s$, $r=s$ and $r>s$, one can easily obtain the result.}
 \end{proof}

%The next theorem shows that the converse of Theorem~\ref{adding} is true.

%\begin{thm}\label{converse adding} Let $n_1,m_1,\ldots,n_k,m_k$ and $x_1,\ldots,x_t$ be some real numbers. If $\{n_1,\ldots,n_k,x_1,\ldots,x_t\}\succeq \{m_1,\ldots,m_k,x_1,\ldots,x_t\}$, then $\{n_1,\ldots,n_k\}\succeq\{m_1,\ldots,m_k\}$. Moreover if $\{n_1,\ldots,n_k,x_1,\ldots,x_t\}\succeq_d \{m_1,\ldots,m_k,x_1,\ldots,x_t\}$, then
%$\{n_1,\ldots,n_k\}\succeq_d\{m_1,\ldots,m_k\}$.
%\end{thm}
%\begin{proof}{It suffices to prove theorem for $t=1$. Suppose that $\{n_1,\ldots,n_k,x_1\}\succeq \{m_1,\ldots,m_k,x_1\}$. We show that $\{n_1,\ldots,n_k\}\succeq\{m_1,\ldots,m_k\}$. If $\{n_1,\ldots,n_k\}=\{m_1,\ldots,m_k\}$, then there is nothing to prove. Let $\{n_1,\ldots,n_k\}\neq\{m_1,\ldots,m_k\}$. By  contradiction suppose that $\{m_1,\ldots,m_k\}\succeq \{n_1,\ldots,n_k\}$. By Theorem~\ref{adding}, $\{m_1,\ldots,m_k,x_1\}\succeq \{n_1,\ldots,n_k,x_1\}$. This shows that $\{n_1,\ldots,n_k\}=\{m_1,\ldots,m_k\}$, a contradiction.

%Now, we prove the second part. Let $\{n_1,\ldots,n_k,x_1\}\succeq_d \{m_1,\ldots,m_k,x_1\}$. We show that $\{n_1,\ldots,n_k\}\succeq_d \{m_1,\ldots,m_k\}$. We can assume that $n_1\geq\cdots\geq n_k$ and $m_1\geq\cdots\geq m_k$. Suppose that $n_1\geq \cdots\geq n_r\geq x_1\geq n_{r+1}\geq\cdots\geq n_k$ and $m_1\geq \cdots\geq m_s\geq x_1\geq m_{s+1}\geq\cdots\geq m_k$, for some $r,s\in\{1,\ldots,k\}$. Considering the cases $r<s$, $r=s$ and $r>s$, one can easily obtain the result.
}
%\end{proof}

{\rm Let $X, Y\in{\cal A}_n$, where ${\cal A}_n=\{(x_1,\ldots,x_n)\in\mathbb{R}^n:\,x_1\geq x_2\geq\cdots\geq x_n\}$.
 Suppose that $e_j=(\underbrace{0,\ldots,0}_{j-1},1,\underbrace{0,\ldots,0}_{n-j})$ and $e_{jk}=e_k-e_j$. We say that $X$ is {\it convertible to} $Y$ if there is a sequence $X=Y_0\preceq Y_1\preceq\cdots\preceq Y_t=Y$ in ${\cal A}_n$, such that for every $i\in\{1,\ldots,t\}$, $Y_i=Y_{i-1}-e_{jk}$ for some $k>j$. For example $X=(9,9,6,6)$ is convertible to $Y=(10,8,7,5)$ ( because $X=(9,9,6,6)\preceq (9,9,7,5)\preceq(10,8,7,5)$ and we have   $(9,9,7,5)=(9,9,6,6)-e_{34}$ and $(10,8,7,5)=(9,9,7,5)-e_{12}$). On the other hand one can easily see that $(8,8,4)$ is not convertible to $(10,5,5)$. Also $(\frac{1}{2},\frac{1}{2})$ is not convertible to $(1,0)$.  In sequel we obtain the sufficient and necessary condition for the convertibility. First we prove some lemmas.}

\begin{lem}\label{m_1,m_2,...,m_2} Let $m_1,m_2,n_1,\ldots,n_{j+1}$ be  some integers. Let $n_1\geq m_1+1$ and $(n_1,\ldots,n_{j+1})\succeq_d(m_1,\underbrace{m_2,\ldots,m_2}_{j})$. Then $(n_1,\ldots,n_{j+1})\succeq_d(m_1+1,\underbrace{m_2,\ldots,m_2}_{j-1},m_2-1)$.
\end{lem}
\begin{proof}{ We proceed by induction on $j$. For $j=1$ there is nothing to prove. Now let $j\geq 2$. Since $(n_1,\ldots,n_{j+1})\succeq_d(m_1,\underbrace{m_2,\ldots,m_2}_{j})$, $\sum_{i=1}^jn_i\geq m_1+(j-1)m_2$ and $t=\sum_{i=1}^{j+1}n_i-(m_1+jm_2)\geq 0$.
Let $n_{j+1}'=n_{j+1}-t$. Thus
\begin{equation}\label{nj+1}
\hbox{$\sum_{i=1}^jn_i+n_{j+1}'=m_1+jm_2$,}
\end{equation}
also we have $\sum_{i=1}^jn_i\geq m_1+(j-1)m_2$. This shows that $m_2\geq n_{j+1}'$. If $n_{j+1}'=m_2$, then $n_2\geq\cdots\geq n_{j+1}\geq n_{j+1}'=m_2$. Thus $\sum_{i=1}^jn_i+n_{j+1}'\geq m_1+1+jm_2$. This contradicts the Equation~(\ref{nj+1}). Let $\l=m_2-n_{j+1}'$. Therefore $\l\geq 1$. By the Equation~(\ref{nj+1}) we obtain
 \begin{equation}\label{l}
\hbox{$\sum_{i=1}^jn_i=m_1+(j-1)m_2+\l$.}
\end{equation}
Using the Equality~(\ref{l}) one can easily see that $(n_1,\ldots,n_{j-1},n_j-\l)\succeq_d(m_1,\underbrace{m_2,\ldots,m_2}_{j-1})$. Now, by the induction hypothesis for $j-1$, we have the following inequality
$(n_1,\ldots,n_{j-1},n_j-\l)\succeq_d(m_1+1,\underbrace{m_2,\ldots,m_2}_{j-2},m_2-1)$. Since $\l\geq 1$, by the Equality~(\ref{l}) one can see that
$(n_1,\ldots,n_{j},n_{j+1}')\succeq_d(m_1+1,\underbrace{m_2,\ldots,m_2}_{j-1},m_2-1)$. On the other hand $(n_1,\ldots,n_{j},n_{j+1})\succeq_d (n_1,\ldots,n_{j},n_{j+1}')$, the proof is complete.}
\end{proof}

\begin{lem}\label{e12} Let ${\cal B}_k=\{(z_1,\ldots,z_k)\in\mathbb{Z}^k:\,z_1\geq z_2\geq\cdots\geq z_k\}$. Let $N=(n_1,\ldots,n_k), M=(m_1,\ldots,m_k)\in{\cal B}_k$. Let $n_1\geq m_1+1$ and $N\succeq_{d}M$.
Then $\{n_1,\ldots,n_k\}\succeq_d\{m_1+1,m_2-1,m_3,\ldots,m_k\}\succeq_d\{m_1,\ldots,m_k\}$ (In the other words  $N\succeq_d M'\succeq_d M$, where $M'\in {\cal B}_k$  and $M'=M-e_{12}$).
\end{lem}
\begin{proof}{ By Theorem~\ref{adding}, $M'\succeq M$. Since $m_2\geq m_3$, we consider the following cases,:
\begin{enumerate}
\item[1)] Let $m_2>m_3$. Thus $M'=(m_1+1,m_2-1,m_3,\ldots,m_k)$. So in this case we have $N\succeq_d M'$.
\item[2)] Let $m_2=m_3$. Suppose that $m_2=m_3=\cdots=m_{j+1}$ and $m_{j+1}>m_{j+2}$ for some $j\geq 2$. Thus $M'=(m_1+1,\underbrace{m_2,\ldots,m_2}_{j-1},m_2-1,m_{j+2},\ldots,m_k)$. Now, we show that $N\succeq_dM'$. One can easily see that it suffices to prove that $(n_1,\ldots,n_j)\succeq_d(m_1+1,\underbrace{m_2,\ldots,m_2}_{j-1})$. Since $N\succeq_d M$,
      $(n_1,\ldots,n_{j+1})\succeq_d(m_1,\ldots,m_{j+1})$. So $(n_1,\ldots,n_{j+1})\succeq_d(m_1,\underbrace{m_2,\ldots,m_2}_{j})$. By Lemma~\ref{m_1,m_2,...,m_2} and the fact that $n_1\geq m_1+1$, $(n_1,\ldots,n_j)\succeq_d(m_1+1,\underbrace{m_2,\ldots,m_2}_{j-1})$. The proof is complete.
\end{enumerate}
}
\end{proof}

{\rm The following result is a direct consequence of Lemma~\ref{e12}.}

\begin{cor}\label{-e12} Let ${\cal B}_n=\{(z_1,\ldots,z_n)\in\mathbb{Z}^k:\,z_1\geq z_2\geq\cdots\geq z_n\}$. Let $X=(x_1,\ldots,x_n), Y=(y_1,\ldots,y_n)\in{\cal B}_n$. Let $x_1=y_1,\ldots,x_j=y_j$ and $y_{j+1}\geq x_{j+1}+1$. If  $Y\succeq_{d}X$, then $\{y_1,\ldots,y_n\}\succeq_d\{x_1,\ldots,x_j,x_{j+1}+1,x_{j+2}-1,x_{j+3},\ldots,x_n\}\succeq_d\{x_1,\ldots,x_n\}$.
\end{cor}

{\rm Now, we are in a position to prove the main theorem of this section. The following theorem states the necessary and sufficient condition for convertibility.}

\begin{thm}\label{convert} Let ${\cal B}_n=\{(z_1,\ldots,z_n)\in\mathbb{Z}^n:\,z_1\geq z_2\geq\cdots\geq z_n\}$. Let $X,Y\in {\cal B}_n$. Then $X$
is  convertible to $Y$ if and only if $Y\succeq_d X$.
\end{thm}
\begin{proof}{ First suppose that $X$
is convertible to $Y$. Thus there is a sequence $X=Y_0\preceq Y_1\preceq\cdots\preceq Y_t=Y$ in ${\cal B}_n$, such that for every $i\in\{1,\ldots,t\}$, $Y_i=Y_{i-1}-e_{jk}$ for some $k>j$. To show that $Y\succeq_d X$, it suffices to prove that $Y_i\succeq_d Y_{i-1}$, for $i=1,\ldots,t$. Let $n\geq k>j\geq 1$ and $Y_i=Y_{i-1}-e_{jk}$. Suppose that $Y_{i-1}=(y_1,\ldots,y_n)$. Thus the components of $Y_i$ are $y_1,\ldots,y_{j-1},y_j+1,y_{j+1},\ldots,y_{k-1},y_k-1,y_{k+1},\ldots,y_n$. Trivially $\{y_{j}+1,y_{k}-1\}\succeq_d\{y_j,y_k\}$. By Theorem~\ref{adding}, $\{y_{j}+1,y_{k}-1\}\cup B
\succeq_d\{y_j,y_k\}\cup B$, where $B=\{y_1,\ldots,y_n\}\setminus\{y_j,y_k\}$. This shows that $Y_i\succeq_d Y_{i-1}$. This completes the proof of the first part.

Now, assume that $Y\succeq_d X$. We show that $X$ is  convertible to $Y$. Let $X=(x_1,\ldots,x_n)$ and $Y=(y_1,\ldots,y_n)$. By Theorem~\ref{dominate}, $Y\succeq X$. Suppose that $y_1=x_1,\ldots,y_j=x_j$ and $y_{j+1}>x_{j+1}$. Thus $y_{j+1}\geq x_{j+1}+1$. Assume that $y_{j+1}=x_{j+1}+\l$. For $i=1,\ldots,\l$, let $X_i=X_{i-1}-e_{(j+1)(j+2)}\in{\cal B}_n$ and $X_0=X$. By Corollary~\ref{e12}, $Y\succeq_dX_i\succeq_dX_{i-1}$. Thus we have $Y\succeq_dX_{\l}\succeq_dX_{\l-1}\succeq_d\cdots\succeq_dX_0$. Note that the first $j+1$ components of $Y$ and $X_{\l}$ are the same. Now, applying this procedure for the pairs $\{Y,X_{\l}\}, \{X_{\l},X_{\l-1}\},\ldots$, we conclude that $X$ is convertible to $Y$.}
\end{proof}

\section{The family of starlike trees and the ordering $\succeq$}

{\rm In this section we investigate about Conjecture~\ref{total-tree} and show that this conjecture is true for the family of starlike trees.

Let ${\cal T}_{n,k}=\{T(n_1,\ldots,n_k):\,n_1\geq\cdots\geq n_k\geq 2\,\,\,{\rm and}\,\sum_{i=1}^kn_i=n\}$. We guess that $({\cal T}_{n,k},\succeq)$ is a totaly ordered set. More precisely we have the following conjecture.}

\begin{conj}\label{starlike and d_G} Let $n_1\geq\cdots\geq n_k\geq 2$ and $m_1\geq\cdots\geq m_k\geq 2$. Let $T_1$ and $T_2$ be two trees of order $n$ and maximum degree $k$ such that $T_1=T(n_1,\ldots,n_k)$ and $T_2=T(m_1,\ldots,m_k)$. Then $(n_1,\ldots,n_k)\succ(m_1,\ldots,m_k)$ if and only if
 $T(m_1,\ldots,m_k)\succ T(n_1,\ldots,n_k)$.
\end{conj}

{\rm In the following theorem we show that Conjecture~\ref{starlike and d_G} is valid for some family of starlike trees.}

\begin{thm}\label{main theorem} Let $n_1\geq\cdots\geq n_k\geq 2$ and $m_1\geq\cdots\geq m_k\geq 2$. Let $T_1$ and $T_2$ be two trees of order $n$ and maximum degree $k$ such that $T_1=T(n_1,\ldots,n_k)$ and $T_2=T(m_1,\ldots,m_k)$. If $(n_1,\ldots,n_k)\succeq_d(m_1,\ldots,m_k)$, then
 $T(m_1,\ldots,m_k)\succeq T(n_1,\ldots,n_k)$. Moreover, if $(n_1,\ldots,n_k)\succ_d(m_1,\ldots,m_k)$, then $T(m_1,\ldots,m_k)\succ T(n_1,\ldots,n_k)$.
\end{thm}
%Let $${\cal C}_{n,k}=\{(w_1,\ldots,w_k)\in\mathbb{N}^n:\,w_1\geq w_2\geq\cdots\geq w_k\geq 2\, {\rm and}\, \sum_{i=1}^kw_i=n\}.$$
\begin{proof}{
If $k\in\{1,2\}$, then $T_1=T_2=P_n$. So we are done. Suppose that $k\geq 3$. We proceed by induction on $n$. Since  the order of $T(n_1,\ldots,n_k)$ is $\sum_{i=1}^kn_i-k+1$ and $n_1\geq\cdots\geq n_k\geq 2$, then $n\geq k+1$. If $n=k+1$, then $T_1=T_2=T(\underbrace{2,\ldots,2}_{k})$. If $n=k+2$, then $T_1=T_2=T(3,\underbrace{2,\ldots,2}_{k-1})$. So there is nothing to prove. If $n=k+3$, then $T_1,T_2\in\{T(4,\underbrace{2,\ldots,2}_{k-1}),T(3,3,\underbrace{2,\ldots,2}_{k-2})\}$. Thus by Theorem~\ref{TH}, we are done. Now, suppose that $n\geq k+4$.

Let $(n_1,\ldots,n_k)\succeq_d(m_1,\ldots,m_k)$. By Theorem~\ref{dominate}, $(n_1,\ldots,n_k)\succeq(m_1,\ldots,m_k)$. If $m_i=n_i$, for $i=1,\ldots,k$, there is noting to prove. Suppose that $m_1=n_1,\ldots,m_j=n_j$ and $n_{j+1}>m_{j+1}$. By Theorem~\ref{convert}, $(m_1,\ldots,m_k)$ is convertible to $(n_1,\ldots,n_k)$.
Thus to complete the proof it suffices to prove the inequality for $M=(m_1,\ldots,m_k)$ and $M'=(m_1,\ldots,m_k)-e_{(j+1)(j+2)}$. More precisely, let $$T_1=T(m_1,\ldots,m_{j+1},m_{j+2},\ldots,m_k)$$ and $$T_2=T(m_1,\ldots,m_j,m_{j+1}+1,m_{j+2}-1,m_{j+3},\ldots,m_k).$$ We show that $T_1\succ T_2$. Let $u$ and $v$ be the vertices of $T_1$ and $T_2$ with degree $k$, respectively. We consider the following cases (Note that if $m_{j+2}=2$, then $m_{j+2}=\cdots=m_{k}=2$. On the other hand $n=\sum_{i=1}^kn_i+k-1=\sum_{i=1}^km_i+k-1$, thus $\sum_{i=j+1}^kn_i=\sum_{i=j+1}^km_i$. This  contradicts the inequality $n_{j+1}>m_{j+1}$. Thus $m_{j+2}\geq 3$):
\begin{enumerate}
\item[1)] Let $m_{j+2}\geq 5$. We have
$T_1\setminus u=P_{m_1-1}+\cdots+P_{m_k-1}$. Let $V(P_{m_{j+2}-1})=\{u_2,\ldots,u_{m_{j+2}}\}$ and
$E(P_{m_{j+2}-1})=\{u_2u_3,u_3u_4,\ldots,u_{m_{j+2}-1}u_{m_{j+2}}\}$ such that $u_2$ is adjacent to $u$. On the other hand one of the components of $T_2\setminus v$ is the path $P_{m_{j+1}}$. Let $V(P_{m_{j+1}})=\{v_2,\ldots,v_{m_{j+1}+1}\}$ and
$E(P_{m_{j+1}})=\{v_2v_3,v_3v_4,\ldots,v_{m_{j+1}}v_{m_{j+1}+1}\}$ such that $v_2$ is adjacent to $v$. Consider the edges $e=u_{m_{j+2}-1}u_{m_{j+2}}$ and $e'=v_{m_{j+1}}v_{m_{j+1}+1}$. Clearly $T_1\setminus e=T_2\setminus e'$. On the other hand
$$T_1\setminus[e]=T(m_1,\ldots,m_{j+1},m_{j+2}-3,m_{j+3},\ldots,m_k),$$ and $$T_2\setminus[e']=T(m_1,\ldots,m_j,m_{j+1}-2,m_{j+2}-1,m_{j+3},\ldots,m_k).$$
Since $\{m_{j+1},m_{j+1}-3\}\succeq_d\{m_{j+1}-2,m_{j+2}-1\}$, by Theorem~\ref{adding} we conclude that $$\{m_1,\ldots,m_{j+1},m_{j+2}-3,m_{j+3},\ldots,m_k\}\succeq_d\{m_1,\ldots,m_j,m_{j+1}-2,m_{j+2}-1,m_{j+3},\ldots,m_k\}.$$
Thus by the induction hypothesis, $T_2\setminus [e']\succ T_1\setminus [e]$. Using the second part of Theorem~\ref{two part strict}, we obtain that $T_1\succ T_2$.
\item[2)] Let $m_{j+2}=3$. Thus $T_1\setminus u$ has the path $P_2$ as a component. Let $u_2,u_3$ be the vertices of $P_2$, such that $u_2$ is adjacent to $u$. On the other hand one of the components of $T_2\setminus v$ is the path $P_{m_{j+1}}$. Let $V(P_{m_{j+1}})=\{v_2,\ldots,v_{m_{j+1}+1}\}$ and
$E(P_{m_{j+1}})=\{v_2v_3,v_3v_4,\ldots,v_{m_{j+1}}v_{m_{j+1}+1}\}$ such that $v_2$ is adjacent to $v$. Consider the vertices $u_3$ and $v_{m_{j+1}+1}$. Obviously $T_1\setminus u_3=T_2\setminus v_{m_{j+1}+1}$ and we have $T_1\setminus[u_3]=T(m_1,\ldots,m_{j+1},m_{j+3},\ldots,m_{k})$ and $T_2\setminus[v_{m_{j+1}+1}]=T(m_1,\ldots,m_{j},m_{j+1}-1,2,m_{j+3},\ldots,m_k)$. Similar to the first case by choosing suitable edges in the trees $T_1\setminus[u_3]$ and $T_2\setminus[v_{m_{j+1}+1}]$ and applying Theorem~\ref{subgraph}, one can see that $T_2\setminus[v_{m_{j+1}+1}]\succ T_1\setminus[u_3]$. Now by the first part of Theorem~\ref{two part strict}, we conclude that $T_1\succ T_2$.
\item[3)] Let $m_{j+2}=4$. Similar to the previous cases one can obtain $T_1\succ T_2$.
\end{enumerate}
}
\end{proof}

{\rm By Remark~\ref{root greater} we have the following result.}

\begin{cor}\label{cor main theorem} Let $n_1\geq\cdots\geq n_k\geq 2$ and $m_1\geq\cdots\geq m_k\geq 2$. Let $T_1$ and $T_2$ be two trees of order $n$ and maximum degree $k$ such that $T_1=T(n_1,\ldots,n_k)$ and $T_2=T(m_1,\ldots,m_k)$. If $(n_1,\ldots,n_k)\succeq_d(m_1,\ldots,m_k)$, then
 $\xi(T_2)\geq\xi(T_1)$.
\end{cor}

{\rm The following result shows that among a special family of trees, starlike trees are uniquely determined by their independence polynomials.}

\begin{cor}\label{uniqely determined} Let $n_1\geq\cdots\geq n_k\geq 2$ and $m_1\geq\cdots\geq m_k\geq 2$. Let $T_1$ and $T_2$ be two trees of order $n$ and maximum degree $k$ such that $T_1=T(n_1,\ldots,n_k)$ and $T_2=T(m_1,\ldots,m_k)$. If $(n_1,\ldots,n_k)\succ_d(m_1,\ldots,m_k)$, then
 $I(T_1,x)\neq I(T_2,x)$.
\end{cor}

{\rm We finish this paper by some questions as follows. We think that there is a relation between the ordering $\succeq$ and the degree sequence of vertices of graphs. More precisely, {\rm let $G$ be a graph of order $n$. Let $d_1\geq d_2\geq\cdots\geq d_n$ be the degree sequence of $G$. By $D_G$ we mean $(d_1,d_2,\ldots,d_n)$. Let $T\neq T_{n,k}, H_{n,k}$ be a tree of order $n$ and maximum degree $k$. It is not hard to see that $D_{H_{n,k}}\succ D_T\succ D_{T_{n,k}}$. As we proved in Theorem~\ref{TH},
$H_{n,k}\succ T\succ T_{n,k}$. It motivates us to pose the following questions:}

\begin{question}\label{d_G} Let
$T_1$ and $T_2$ be two trees of order $n$. Let $D_{T_1}\succ D_{T_2}$. Is it true that $T_1\succ T_2$?
\end{question}

\begin{question}\label{d and I} Let
$T_1$ and $T_2$ be two trees of order $n$. Suppose that $I(T_1,x)=I(T_2,x)$. Is it true that $D_{T_1}=D_{T_2}$?
\end{question}

\noindent{\bf Acknowledgements.} This research
was in part supported by a grant (No. 91050013) from School of
Mathematics, Institute for Research in Fundamental Sciences (IPM). The research of the author is partially supported by the Center of Excellence for Mathematics, University of Isfahan.

\end{document}